\documentclass[12pt]{article}

\usepackage{amsmath, epsfig, cite}
\usepackage{amssymb}
\usepackage{amsfonts}
\usepackage{latexsym}
\usepackage{amsthm}

\makeatletter
\renewcommand{\@seccntformat}[1]{{\csname the#1\endcsname}.\hspace{.5em}}
\makeatother

\newtheorem{thm}{Theorem}[section]

\newtheorem{prob}[thm]{Problem}
\newtheorem{cor}[thm]{Corollary}
\newtheorem{conj}[thm]{Conjecture}
\newtheorem{lem}[thm]{Lemma}

\newcommand{\pf}{\noindent{\it Proof.} }

\def\Q{\mathbb Q}

\renewcommand{\qed}{\hfill$\Box$\medskip}

\numberwithin{equation}{section}

\begin{document}

\renewcommand{\thefootnote}{*}

\begin{center}
{\Large\bf Some $q$-analogues of (super)congruences of Beukers,\\[5pt] Van Hamme and Rodriguez-Villegas}
\end{center}

\vskip 2mm \centerline{Victor J. W. Guo$^1$  and Jiang Zeng$^{2}$}
\begin{center}
{\footnotesize $^1$Department of Mathematics, Shanghai Key Laboratory of PMMP,
East China Normal University,\\ 500 Dongchuan Road, Shanghai 200241,
 People's Republic of China\\
{\tt jwguo@math.ecnu.edu.cn,\quad htap://math.ecnu.edu.cn/\textasciitilde{jwguo}}\\[10pt]
$^2$Universit\'e de Lyon; Universit\'e Lyon 1; Institut Camille
Jordan, UMR 5208 du CNRS;\\ 43, boulevard du 11 novembre 1918,
F-69622 Villeurbanne Cedex, France\\
{\tt zeng@math.univ-lyon1.fr,\quad
htap://math.univ-lyon1.fr/\textasciitilde{zeng}} }
\end{center}

\vskip 0.7cm \noindent{\small{\bf Abstract.}}
For any odd prime $p$ we obtain $q$-analogues of Van Hamme's supercongruence:
$$
\sum_{k=0}^{\frac{p-1}{2}}{2k\choose k}^3\frac{1}{64^k} \equiv 0
\pmod{p^2} \quad\text{for}\quad p\equiv 3\pmod 4,
$$
and Rodriguez-Villegas' Beukers-like supercongruences involving products of three binomial coefficients.
For example, we prove that
\begin{align*}
\sum_{k=0}^{\frac{p-1}{2}} {2k\brack k}_{q^2}^3 \frac{q^{2k}}{(-q^2;q^2)_k^2 (-q;q)_{2k}^2}
&\equiv 0\pmod{[p]^2}  \quad\text{for}\quad p\equiv 3\pmod 4,  \\
\sum_{k=0}^{\frac{p-1}{2}}{2k\brack k}_{q^3}\frac{(q;q^3)_k (q^{2};q^3)_{k} q^{3k} }{ (q^{6};q^{6})_k^2 }
&\equiv 0 \pmod{[p]^2}\quad\text{for}\quad p\equiv 2\pmod{3},
\end{align*}
where $[p]=1+q+\cdots+q^{p-1}$, $(a;q)_n=(1-a)(1-aq)\cdots(1-aq^{n-1})$, and ${n\brack k}_q$ denotes the $q$-binomial coefficient.
Actually, our results
give  $q$-analogues of Z.-H. Sun's and Z.-W. Sun's generalizations of the above Beukers-like supercongruences.
Our proof uses the theory of basic hypergeometric series including a new $q$-Clausen-type summation formula.

\vskip 3mm \noindent {\it Keywords}: congruences, supercongruences, basic hypergeometric series, $q$-binomial theorem, $q$-Chu-Vandermonde formula, the Lagrange interpolation,
the Newton interpolation

\vskip 3mm \noindent {\it 2000 Mathematics Subject Classifications}: Primary 11B65, Secondary 05A10, 05A30

\section{Introduction}
In 1985, Beukers \cite{Beukers} made the following conjecture: for any odd prime $p$,
\begin{align}
&\hskip -2mm\sum_{k=0}^{\frac{p-1}{2}}(-1)^k{\frac{p-1}{2}\choose k} {2k\choose k}^2\frac{1}{16^k} \notag \\
&\equiv
\begin{cases}
4x^2-2p \pmod{p^2}, &\text{if $p=x^2+y^2$ with $x$ odd,}\\
0\pmod{p^2}, &\text{if $p\equiv 3\pmod 4$}
\end{cases}  \pmod{p^2}.\label{eq:beukers-0}
\end{align}
Beukers himself proved this congruence modulo $p$, which is equivalent to
\begin{align}
\sum_{k=0}^{\frac{p-1}{2}}{2k\choose k}^3\frac{1}{64^k}
\equiv
\begin{cases}
4x^2 \pmod{p}, &\text{if $p=x^2+y^2$ with $x$ odd,}\\
0\pmod{p}, &\text{if $p\equiv 3\pmod 4$}
\end{cases}  \pmod{p}.\label{eq:beukers}
\end{align}
 Complete proofs of \eqref{eq:beukers-0}
have been given by Ishikawa \cite{Ishikawa}, Ahlgren \cite{Ahlgren}, and Mortenson \cite{Mortenson5}.
It is interesting that Van Hamme \cite{Hamme} proved the following company congruence of \eqref{eq:beukers-0}.
\begin{align}
\sum_{k=0}^{\frac{p-1}{2}}{2k\choose k}^3\frac{1}{64^k}
\equiv
\begin{cases}
4x^2-2p \pmod{p^2}, &\text{if $p=x^2+y^2$ with $x$ odd,}\\
0\pmod{p^2}, &\text{if $p\equiv 3\pmod 4$}
\end{cases}  \pmod{p^2}.\label{eq:beukers-vanhamme}
\end{align}
Finite fields analogues of classical hypergeometric series \cite{Greene}
play important parts in \cite{Ahlgren, AO, Mortenson5}.

Motivated by his joint work \cite{COR} with Candelas and de la Ossa on Calabi-Yau manifolds over finite fields,
Rodriguez-Villegas \cite{RV} discovered numerically some Beukers-like supercongruences, such as, for
any prime $p>3$,
\begin{align}
\sum_{k=0}^{p-1}{3k\choose k}{2k\choose k}^2\frac{1}{108^k}
&\equiv 0\pmod{p^2}\quad\text{if}\quad p\equiv 2\pmod{3}, \label{eq:RV-B1}\\
\sum_{k=0}^{p-1}{4k\choose k}{2k\choose k}^2\frac{1}{256^k}
&\equiv 0\pmod{p^2}\quad\text{if}\quad p\equiv 5,7\pmod{8}, \label{eq:RV-B2} \\
\sum_{k=0}^{p-1}{6k\choose 3k}{3k\choose k}{2k\choose k}\frac{1}{1728^k}
&\equiv 0\pmod{p^2}\quad\text{if}\quad p\equiv 3\pmod{4}.  \label{eq:RV-B3}
\end{align}
Note that Rodriguez-Villegas \cite{RV} made  conjectures concerning the numbers in \eqref{eq:RV-B1}--\eqref{eq:RV-B3} for general $p$,
which were finally proved by Mortenson \cite{Mortenson5} and Z.-W. Sun \cite{Sun0}.
Recently, by using the properties of generalized Legendre polynomials,
Z.-H. Sun \cite{SunZH2}, among other things, proved the following result.
\begin{thm} {{\rm\cite[Theorem~2.5]{SunZH2}}}
Let $p$ be an odd prime and let $a$ be a $p$-adic integer with $\langle a\rangle_p\equiv 1\pmod 2$. Then
\begin{align}
\sum_{k=0}^{p-1}{2k\choose k}{a\choose k}{-1-a\choose k}\frac{1}{4^k}\equiv 0 \pmod{p^2}, \label{eq:sun-lengendre-1}
\end{align}
where $\langle a \rangle_p$ denotes the least nonnegative residue of $a$ modulo $p$.
\end{thm}

It is not difficult to see that, taking $a=-\frac{1}{2}, -\frac{1}{3}, -\frac{1}{4}, -\frac{1}{6}$ in \eqref{eq:sun-lengendre-1},
we obtain the $p\equiv 3\pmod{4}$ case of \eqref{eq:beukers}, and \eqref{eq:RV-B1}--\eqref{eq:RV-B3}.
The first aim of this paper is to give a $q$-analogue of \eqref{eq:beukers-vanhamme} and \eqref{eq:sun-lengendre-1}.

On the other hand, for any prime $p\geqslant 5$, Mortenson \cite{Mortenson1,Mortenson4} proved  the following four
supercongruences conjectured by Rodriguez-Villegas\cite[(36)]{RV}:
\begin{align}
\sum_{k=0}^{p-1}{2k\choose k}^2\frac{1}{16^k}
&\equiv \left(\frac{-1}{p}\right)\pmod{p^2}, \label{eq:RV1} \\
\sum_{k=0}^{p-1}{3k\choose 2k}{2k\choose k}\frac{1}{27^k}
&\equiv \left(\frac{-3}{p}\right)\pmod{p^2}, \label{eq:RV2} \\
\sum_{k=0}^{p-1}{4k\choose 2k}{2k\choose k}\frac{1}{64^k}
&\equiv \left(\frac{-2}{p}\right)\pmod{p^2}, \label{eq:RV3} \\
\sum_{k=0}^{p-1}{6k\choose 3k}{3k\choose k}\frac{1}{432^k}
&\equiv \left(\frac{-1}{p}\right)\pmod{p^2}, \label{eq:RV4}
\end{align}
where $(\frac{\cdot}{p})$ denotes the Legendre symbol modulo $p$. Z.-H. Sun~\cite{SunZH2} gave
an elementary proof of \eqref{eq:RV1}--\eqref{eq:RV4} by using generalized Legendre polynomials. See also
\cite{SunZH1,Tauraso,Tauraso1} for several simple proofs of \eqref{eq:RV1}.
Z.-W. Sun \cite[(1.4)]{Sun13} obtained the following generalization of \eqref{eq:RV1}:
\begin{align}
\sum_{k=0}^{\frac{p-1}{2}}\frac{{2k\choose k}{2k+2s\choose k+s}}{4^{2k+s}}
\equiv \left(\frac{-1}{p}\right) \pmod{p^2},\quad\text{for}\ 0\leqslant s\leqslant \frac{p-1}{2}, \label{eq:sun13}
\end{align}
where $(\frac{\cdot}{p})$ denotes the Legendre symbol modulo $p$. Note that McCarthy and Osburn \cite{MO}
have studied some related interesting supercongruences.

Recall that  the {\it $q$-shifted factorials} are defined by $(a;q)_0=1$ and
$$(a;q)_n=(1-a)(1-aq)\cdots(1-aq^{n-1})\ \text{for}\ n=1,2,\ldots,$$
and the {\it $q$-integer} is defined as $[p]=\frac{1-q^p}{1-q}$.
The {\it $q$-binomial coefficients} ${n\brack k}$ are then defined by
$$
{n\brack k}={n\brack k}_q
=\begin{cases}
\displaystyle\frac{(q^{n-k+1};q)_k}{(q;q)_k}, &\text{if $0\leqslant k\leqslant n$},\\[10pt]
0,&\text{otherwise.}
\end{cases}
$$

In the last decade, several authors have studied $q$-analogues of congruences and supercongruences,
see \cite{Andrews99,Pan,SP,Straub,Tauraso2}.
Throughout the paper we will often use the fact that for any prime $p$, the $q$-integer
$[p]$ is always an irreducible polynomial in $\Q[q]$.
Namely, $\Q[q]/[p]$ is a field.
Therefore, rational functions $a(q)/b(q)$ are
well defined modulo $[p]$ or $[p]^r$ ($r\geqslant 1$) on condition that $b(q)$ is relatively prime to $[p]$.
In a previous paper~\cite{GZ}, we give
some $q$-analogues of \eqref{eq:RV1} and
partial $q$-analogues of \eqref{eq:RV2}--\eqref{eq:RV4} such as
\begin{align}
\sum_{k=0}^{p-1}\frac{(q;q^2)_k^2}{(q^2;q^2)_k^2}
&\equiv \left(\frac{-1}{p}\right)q^{\frac{1-p^2}{4}}\pmod{[p]^2}, \label{eq:q-RV1}  \\
\sum_{k=0}^{p-1}\frac{(q;q^3)_k (q^2;q^3)_k}{(q^3;q^3)_k^2}
&\equiv \left(\frac{-3}{p}\right)q^{\frac{1-p^2}{3}}\pmod{[p]}, \label{eq:q-RV2} \\
\sum_{k=0}^{p-1}\frac{(q;q^4)_k (q^3;q^4)_k}{(q^4;q^4)_k^2}
&\equiv \left(\frac{-2}{p}\right)q^{\frac{3(1-p^2)}{8}}\pmod{[p]}, \label{eq:q-RV3} \\
\sum_{k=0}^{p-1}\frac{(q;q^6)_k (q^5;q^6)_k}{(q^6;q^6)_k^2}
&\equiv \left(\frac{-1}{p}\right)q^{\frac{5(1-p^2)}{12}}\pmod{[p]}. \label{eq:q-RV4}
\end{align}

The second aim of this paper is to give a $q$-analogue of \eqref{eq:sun13}  and a
generalization of \eqref{eq:q-RV1}--\eqref{eq:q-RV4}.

\section{The main results}
Motivated by Z.-W. Sun's generalization \cite[Theorem 1.1(i)]{Sun0} of Van Hamme' congruence \eqref{eq:beukers-vanhamme}:
\begin{align}
\sum_{k=0}^{\frac{p-1}{2}}{2k\choose k}^2{2k\choose k+s}\frac{1}{64^k}
\equiv 0 \pmod{p^2},  \label{eq:BS}
\end{align}
where $0\leqslant s\leqslant \frac{p-1}{2}$ and $s\equiv\frac{p+1}{2}\pmod{2}$, we give a unified $q$-analogue
of \eqref{eq:beukers-vanhamme} and \eqref{eq:BS} as follows.
\begin{thm}\label{thm:main1}
Let $p$ be an odd prime and $0\leqslant s\leqslant \frac{p-1}{2}$. Then we have the following congruence modulo $[p]^2${\rm:}
\begin{align}
&\hskip -2mm\sum_{k=0}^{\frac{p-1}{2}} {2k\brack k}_{q^2}^2{2k\brack k+s}_{q^2} \frac{q^{2k}}{(-q^2;q^2)_k^2 (-q;q)_{2k}^2}  \nonumber\\
&\equiv \begin{cases}
\displaystyle (-1)^s q^{\frac{p-1}{2}-s^2}
{\frac{p-1}{2}\brack \frac {p-2s-1}{4}}_{q^4}^2
\frac{(q^2;q^2)_{\frac{p-2s-1}{2}}(q^2;q^2)_{\frac{p+2s-1}{2}}}{(q^4;q^4)_{\frac{p-1}{2}}^2},
&\text{if $s\equiv \frac{p-1}{2}\pmod 2$,} \\[5pt]
0, &\text{otherwise.}
\end{cases}\label{eq:2.2}
\end{align}
\end{thm}

Letting $s=0$ in \eqref{eq:2.2}, we immediately get the following neat $q$-analogue
of \eqref{eq:beukers}.
\begin{cor}
Let $p$ be an odd prime. Then we have the following congruence{\rm:}
\begin{align}
&\hskip -2mm\sum_{k=0}^{\frac{p-1}{2}} {2k\brack k}_{q^2}^3 \frac{q^{2k}}{(-q^2;q^2)_k^2 (-q;q)_{2k}^2}  \nonumber\\
&\equiv \begin{cases}
\displaystyle q^{\frac{p-1}{2}}
{\frac{p-1}{2}\brack \frac {p-1}{4}}_{q^4}^2\frac{1}{(-q^2;q^2)_{\frac{p-1}{2}}^2},
&\text{if $p\equiv 1\pmod 4$,} \\[5pt]
0, &\text{otherwise,}
\end{cases}  \pmod{[p]^2}.  \label{eq:beukers-q}
\end{align}
\end{cor}

Letting $q\to 1$ in \eqref{eq:beukers-q}, and using
the congruence (see \cite[Lemma 3.4]{SunZH1.5})
\begin{align}
{\frac{p-1}{2}\choose \frac {p-1}{4}}^2\frac{1}{2^{p-1}}\equiv 4x^2-2p\pmod{p^2}, \label{eq:CDE-Sun}
\end{align}
where $p\equiv 1\pmod{4}$ is a prime and $p=x^2+y^2$ with $x\equiv 1\pmod{4}$, we obtain
Van Hamme's congruence \eqref{eq:beukers}.  Note that \eqref{eq:CDE-Sun}
is easily proved by the Beukers-Chowla-Dwork-Evans congruence \cite{CDE,Pan2}:
\begin{align*}
{\frac{p-1}{2}\choose \frac{p-1}{4}}\equiv\frac{2^{p-1}+1}{2}\left(2x-\frac{p}{2x}\right) \pmod{p^2},
\end{align*}
where the numbers $p$ and $x$ are the same as in \eqref{eq:CDE-Sun}.

\medskip
\noindent{\it Remark.} Let $p$ be an odd prime of the form $4k+3$. Then by the antisymmetry of the $k$-th
and $(\frac{p-1}{2}-k)$-th terms modulo $[p]$, we can easily prove that
\begin{align*}
&\hskip -2mm\sum_{k=0}^{\frac{p-1}{2}} {2k\brack k}_{q^2}^3 \frac{q^{k-2k^2}}{(-q^2;q^2)_k^2 (-q;q)_{2k}^2}  \equiv 0 \pmod{[p]}.
\end{align*}

Our second result is a unified $q$-analogue of Z.-H. Sun's congruence \eqref{eq:sun-lengendre-1} and Z.-W. Sun's generalization of
\eqref{eq:RV-B1}--\eqref{eq:RV-B3}.
\begin{thm}\label{thm:2k2k2k-2}
Let $p$ be an odd prime and $m$, $r$ two positive integers with $p\nmid m$.
Let $s\leqslant\min\{\langle -\frac{r}{m}\rangle_p, \langle -\frac{m-r}{m}\rangle_p\}$ be
a nonnegative integer. If $\langle -\frac{r}{m}\rangle_p\equiv s+1\pmod 2$, then
\begin{align}
\sum_{k=s}^{\frac{p-1}{2}}\frac{(q^m;q^m)_{2k}(q^r;q^m)_k (q^{m-r};q^m)_{k} q^{mk} }{ (q^m;q^m)_{k-s}(q^m;q^m)_{k+s} (q^{2m};q^{2m})_k^2 }
&\equiv 0 \pmod{[p]^2}, \label{eq:new-2mpp-002}  \\
\sum_{k=s}^{p-1}\frac{(q^m;q^m)_{2k}(q^r;q^m)_k (q^{m-r};q^m)_{k} q^{mk} }{ (q^m;q^m)_{k-s}(q^m;q^m)_{k+s} (q^{2m};q^{2m})_k^2 }
&\equiv 0 \pmod{[p]^2}.  \label{eq:new-2mpp}
\end{align}
If $\langle -\frac{r}{m}\rangle_p\equiv s\pmod 2$, then the following congruence holds modulo $[p]${\rm:}
\begin{align}
&\hskip -2mm
\sum_{k=s}^{\frac{p-1}{2}}\frac{(q^m;q^m)_{2k}(q^r;q^m)_k (q^{m-r};q^m)_{k} q^{mk} }{ (q^m;q^m)_{k-s}(q^m;q^m)_{k+s} (q^{2m};q^{2m})_k^2 }
\notag\\
&\equiv\frac{q^{(s+\frac{p-1}{2})m} (q^m; q^{2m})_{\frac{\langle -\frac{r}{m}\rangle_p-s}{2}}
(q^m; q^{2m})_{\frac{\langle -\frac{m-r}{m}\rangle_p-s}{2}}
(q^{-m\langle -\frac{r}{m}\rangle_p};q^m)_s (q^{-m\langle -\frac{m-r}{m}\rangle_p};q^m)_s }
 {(q^{2m}; q^{2m})_{\frac{\langle-\frac{r}{m}\rangle_p+s}{2}} (q^{2m}; q^{2m})_{\frac{\langle-\frac{m-r}{m}\rangle_p+s}{2}}}. \label{eq:new-2mpp-003}
\end{align}
\end{thm}

Letting $s=0$, $-\frac{r}{m}=a$ and $q\to 1$ in \eqref{eq:new-2mpp}, we obtain \eqref{eq:sun-lengendre-1}.
On the other hand, it is not difficult to see that (see \cite{GZ}), for any prime $p\geqslant 5$,
\begin{align}\label{eq:values}
(-1)^{\langle -\frac{1}{3}\rangle_p}
=\left(\frac{-3}{p}\right),\quad
(-1)^{\langle -\frac{1}{4}\rangle_p}
=\left(\frac{-2}{p}\right), \quad
(-1)^{\langle -\frac{1}{6}\rangle_p}
=\left(\frac{-1}{p}\right).
\end{align}
Taking $r=1$ and $m=3,4,6$ in \eqref{eq:new-2mpp-002}, we obtain
\begin{cor}
Let $\geqslant 5$ be a prime and let $s$ be a nonnegative integer. Then there hold the following congruences modulo $[p]^2${\rm:}
\begin{align*}
\sum_{k=s}^{\frac{p-1}{2}}{2k\brack k+s}_{q^3}\frac{(q;q^3)_k (q^{2};q^3)_{k} q^{3k} }{ (q^{6};q^{6})_k^2 }
&\equiv 0,\ \text{if }s\leqslant\frac{p-1}{3}\ \text{and }s\equiv \frac{1+(\frac{-3}{p})}{2}\pmod{2},\\
\sum_{k=s}^{\frac{p-1}{2}}{2k\brack k+s}_{q^4}\frac{(q;q^4)_k (q^{3};q^4)_{k} q^{4k} }{ (q^{8};q^{8})_k^2 }
&\equiv 0, \ \text{if }s\leqslant\frac{p-1}{4}\ \text{and } s\equiv \frac{1+(\frac{-2}{p})}{2}\pmod{2},\\
\sum_{k=s}^{\frac{p-1}{2}}{2k\brack k+s}_{q^6}\frac{(q;q^6)_k (q^{5};q^6)_{k} q^{6k} }{ (q^{12};q^{12})_k^2 }
&\equiv 0,\ \text{if }s\leqslant\frac{p-1}{6}\ \text{and }s\equiv \frac{1+(\frac{-1}{p})}{2}\pmod{2}.
\end{align*}
\end{cor}

The proof of Theorem~\ref{thm:2k2k2k-2} is based on the following highly non-trivial $q$-Clausen-type summation formula, which seems new and
interesting in its own right.
\begin{thm}\label{thm:important-thm}
Let $n$ and $s$ be nonnegative integers with $s\leqslant n$. Then
\begin{align}
&\hskip -3mm \left(\sum_{k=s}^{n} \frac{(q^{-2n};q^2)_k (x;q)_k q^k }{(q;q)_{k-s} (q;q)_{k+s}}\right)
  \left(\sum_{k=s}^{n} \frac{(q^{-2n};q^2)_k (q/x;q)_k q^k }{(q;q)_{k-s} (q;q)_{k+s}}\right) \nonumber \\
&=\frac{(-1)^n(q^2;q^2)_n^2 q^{-n^2}}{(q^2;q^2)_{n-s}(q^2;q^2)_{n+s}}
\sum_{k=s}^{n} \frac{(-1)^k(q^2;q^2)_{n+k} (x;q)_k (q/x;q)_k  q^{k^2-2nk}}{(q^2;q^2)_{n-k} (q;q)_{k-s}(q;q)_{k+s} (q;q)_{2k}}.
\label{eq:important-thm}
\end{align}
\end{thm}

Recall that the {\it basic hypergeometric series $_{r+1}\phi_r$} (see \cite{GR}) is defined by
\[
_{r+1}\phi_{r}\left[\!\!\begin{array}{c}
a_1,a_2,\ldots,a_{r+1}\\
b_1,b_2,\ldots,b_{r}
\end{array};\,q, z
\right]
=\sum_{n=0}^{\infty}\frac{(a_1,a_2,\ldots,a_{r+1};\,q)_n z^n}
{(q,b_1,b_2,\ldots,b_{r};\,q)_n},
\]
where $(a_1,\ldots,a_m;q)_n=(a_1;q)_n\cdots (a_m;q)_n$.
Theorem~\ref{thm:important-thm}
is reminiscent to Jackson's $q$-analogue of Clausen's formula:
\begin{align*}
&_{2}\phi_{1}\left[\begin{array}{c}
a,\ b\\
abq^{\frac{1}{2}}
\end{array};\,q,z
\right]
{}_{2}\phi_{1}\left[\begin{array}{c}
a,\ b\\
abq^{\frac{1}{2}}
\end{array};\,q,zq^{\frac{1}{2}}
\right]
={}_{4}\phi_{3}\left[\begin{array}{c}
a,\ b,\ a^{\frac{1}{2}}b^{\frac{1}{2}},\ -a^{\frac{1}{2}}b^{\frac{1}{2}}\\
ab,\ a^{\frac{1}{2}}b^{\frac{1}{2}}q^{1/4},\
-a^{\frac{1}{2}}b^{\frac{1}{2}}q^{1/4}
\end{array};\,q^{\frac{1}{2}},z
\right].  
\end{align*}

We also have the following $q$-analogue of \eqref{eq:sun13}, which  reduces to \eqref{eq:q-RV1} when $s=0$.
\begin{thm}\label{thm:2k2k-r}
Let $p$ be an odd prime and let $0\leqslant s\leqslant \frac{p-1}{2}$. Then
\begin{align}
\sum_{k=0}^{\frac{p-1}{2}}\frac{(q;q^2)_k (q;q^2)_{k+s}}{(q^2;q^2)_k (q^2;q^2)_{k+s} }
&\equiv \left(\frac{-1}{p}\right)q^{\frac{1-p^2}{4}}\pmod{[p]^2}. \label{eq:new-RV1}
\end{align}
\end{thm}

Finally, we have  the following generalization of the previous congruences \eqref{eq:q-RV2}--\eqref{eq:q-RV4}.
\begin{thm}\label{thm:2mp2}
Let $p$ be an odd prime and let $m$, $r$ be positive integers with $p\nmid m$ and $r<m$.
Then for any integer $s$ with $0\leqslant s\leqslant \langle -\frac{m-r}{m}\rangle_p$,
there holds
\begin{align}
\sum_{k=0}^{p-s-1}\frac{(q^r;q^m)_k (q^{m-r};q^m)_{k+s} }{(q^m;q^m)_k (q^m;q^m)_{k+s} }
\equiv (-1)^{\langle -\frac{r}{m}\rangle_p}
q^{\frac{-m\langle -\frac{r}{m}\rangle_p \left(\langle -\frac{r}{m}\rangle_p+1 \right)}{2}} \pmod{[p]}. \label{eq:new-2mp}
\end{align}
In particular, if $p\equiv \pm 1\pmod m$, then
\begin{align}
\sum_{k=0}^{p-s-1}\frac{(q^r;q^m)_k (q^{m-r};q^m)_{k+s} }{(q^m;q^m)_k (q^m;q^m)_{k+s} }
\equiv (-1)^{\langle -\frac{r}{m}\rangle_p}q^{\frac{r(m-r)(1-p^2)}{2m}} \pmod{[p]}. \label{eq:new-2mp-2}
\end{align}
\end{thm}

By \eqref{eq:values}, letting $m=3,4,6$ in \eqref{eq:new-2mp-2} and noticing that
\begin{align*}
(-1)^{\langle -\frac{1}{3}\rangle_p}=(-1)^{\langle -\frac{2}{3}\rangle_p},\quad
(-1)^{\langle -\frac{1}{4}\rangle_p}=(-1)^{\langle -\frac{3}{4}\rangle_p},
\quad
(-1)^{\langle -\frac{1}{6}\rangle_p}=(-1)^{\langle -\frac{5}{6}\rangle_p},
\end{align*}
we obtain
\begin{cor}\label{cor:two}
Let $p>3$ be a prime and $s\geqslant 0$. Then the following congruences hold modulo $[p]${\rm:}
\begin{align}
\sum_{k=0}^{p-s-1}\frac{(q^r;q^3)_k (q^{3-r};q^3)_{k+s}}{(q^3;q^3)_k (q^3;q^3)_{k+s} }
&\equiv \left(\frac{-3}{p}\right)q^{\frac{1-p^2}{4}}\quad \text{for}\   r=1,2,
\text{and}\ s\leqslant\left\langle \frac{r-3}{3}\right\rangle_p, \label{eq:new-RV2}  \\
\sum_{k=0}^{p-s-1}\frac{(q^r;q^4)_k (q^{4-r};q^4)_{k+s}}{(q^4;q^4)_k (q^4;q^4)_{k+s} }
&\equiv \left(\frac{-2}{p}\right)q^{\frac{3(1-p^2)}{8}}\quad \text{for}\   r=1,3,
\text{and}\ s\leqslant\left\langle \frac{r-4}{4}\right\rangle_p,  \label{eq:new-RV3} \\
\sum_{k=0}^{p-s-1}\frac{(q^r;q^6)_k (q^{6-r};q^6)_{k+s}}{(q^6;q^6)_k (q^6;q^6)_{k+s} }
&\equiv \left(\frac{-1}{p}\right)q^{\frac{5(1-p^2)}{12}}\quad \text{for}\   r=1,5,
\text{and}\ s\leqslant\left\langle \frac{r-6}{6}\right\rangle_p. \label{eq:new-RV4}
\end{align}
\end{cor}

\section{Proof of Theorem \ref{thm:main1}}
We first establish two lemmas.
\begin{lem}\label{lem:one}
Let $p$ be an odd prime and $0\leqslant k\leqslant \frac{p-1}{2}$. Then
\begin{align}
{\frac{p-1}{2}+k\brack 2k}_{q^2}
\equiv (-1)^k {2k\brack k}_{q^2}\frac{q^{kp-k^2}}{(-q;q)_{2k}^2}  \pmod{[p]^2}.
\label{eq:lem1}
\end{align}
\end{lem}
\pf Since
\begin{align*}
(1-q^{p-2j+1})(1-q^{p+2j-1})+(1-q^{2j-1})^2q^{p-2j+1}=(1-q^p)^2,
\end{align*}
we have
\begin{align*}
 (1-q^{p-2j+1})(1-q^{p+2j-1}) &\equiv -(1-q^{2j-1})^2 q^{p-2j+1} \pmod{[p]^2}.
\end{align*}
It follows that
\begin{align*}
{\frac{p-1}{2}+k\brack 2k}_{q^2}
&=\frac{\prod_{j=1}^k(1-q^{p-2j+1})(1-q^{p+2j-1})}{(q^2;q^2)_{2k}} \\
&\equiv (-1)^k \frac{\prod_{j=1}^k (1-q^{2j-1})^2 q^{p-2j+1} }{(q^2;q^2)_{2k}}  \\
&= (-1)^k {2k\brack k}_q^2\frac{q^{kp-k^2}}{(-q;q)_{2k}^2} \pmod{[p]^2},
\end{align*}
as desired.  \qed

\medskip
\begin{lem} For nonnegative integers $n$ and $s$ such that $s\leq n$ we have
\begin{align}
\hskip -2mm
&\sum_{k=0}^{n}{n+k\brack 2k}{2k\brack k}{2k\brack k+s} \frac{(-1)^k q^{n-k\choose 2}}{(-q;q)_k^2}  \nonumber\\
&=\begin{cases}
\displaystyle (-1)^s q^{\frac{n^2-s^2}{2}}
{n\brack \frac {n-s}{2}}_{q^2}^2 \frac{(q;q)_{n-s}(q;q)_{n+s}}{(q^2;q^2)_n^2}, &\text{if $n\equiv s\pmod 2$,} \\[5pt]
0, &\text{otherwise.}
\end{cases}
\end{align}\label{eq:key}
\end{lem}
\pf
We may rewrite the left-hand side of \eqref{eq:key} as
\begin{align*}
&\sum_{k=s}^{n}
\frac{(q^{-n};q)_{k}(q^{n+1};q)_{k}(q^{\frac{1}{2}};q)_{k}
(-q^{\frac{1}{2}};q)_{k} q^{k+{n\choose 2}}}
{(q;q)_k (q;q)_{k-s}(q;q)_{k+s}(-q;q)_k}  \nonumber \\
&=
%
\frac{(q^{-n};q)_{s}(q^{n+1};q)_{s}(q;q^2)_{2s} q^{s+{n\choose 2}}} {(q^2;q^2)_s (q;q)_{2s}}
{}_{4}\phi_{3}\left[\!\!\begin{array}{c}
q^{s-n},\, q^{n+s+1},\,q^{s+\frac{1}{2}},\, -q^{s+\frac{1}{2}}\\
q^{s+1},\, -q^{s+1},\, q^{2s+1}\end{array}\!\!;q,q\right].
\end{align*}
The result  then follows from Andrews' terminating $q$-analogue of Watson's formula \cite[(II.17)]{GR}:
\begin{align} \label{eq:andrews}
 {}_{4}\phi_{3}\left[\!\!\begin{array}{c}
q^{-n},\, a^{2}q^{n+1},\,b,\, -b\\
aq,\, -aq,\, b^{2}\end{array}\!\!;q,q\right]
=\begin{cases}
0,&\text{if $n$ is odd},\\[5pt]
\displaystyle\frac{b^{n}(q, a^{2}q^{2}/b^{2}; q^{2})_{n/2}}
{(a^{2}q^{2},\, b^{2}q; q^{2})_{n/2}},& \text{if $n$ is even}
\end{cases}
\end{align}
with the substitution of $n, a$ and $b$ by $n-s, q^s$ and
$q^{s+\frac{1}{2}}$, respectively. \qed
\medskip

\noindent{\it Proof of Theorem {\rm\ref{thm:main1}.}} By the congruence \eqref{eq:lem1}, the left-hand side of \eqref{eq:2.2} is equal to
\begin{align*}
&\hskip -2mm\sum_{k=0}^{\frac{p-1}{2}} {2k\brack k}_{q^2}^2{2k\brack k+s}_{q^2} \frac{q^{2k}}{(-q^2;q^2)_k^2 (-q;q)_{2k}^2} \nonumber \\
&\equiv \sum_{k=0}^{\frac{p-1}{2}}
{\frac{p-1}{2}+k\brack 2k}_{q^2}{2k\brack k}_{q^2}{2k\brack k+s}_{q^2}
\frac{(-1)^k}{(-q^2;q^2)_k^2}q^{k^2+2k-pk}.
\end{align*}
The proof then follows from \eqref{eq:key} with
$n=\frac{p-1}{2}$ and $q\rightarrow q^2$.
\qed

\section{Proof of Theorem~\ref{thm:important-thm}}

We first establish four lemmas to make the proof easier.
The following result can be derived from the Lagrange interpolation formula and
the Newton interpolation formula (see \cite{FL}),  we give a simple proof using the partial fraction decomposition technique as in \cite{Zeng}.

\begin{lem} Let $n$ be a nonnegative integer. Then
\begin{align}\label{simpleLag}
\sum_{k=0}^n (-1)^{n-k} {n\brack k} \frac{(aq^n;q)_k q^{n-k+1\choose 2}}{(a;q)_k (1-xq^{-k})}
&=\frac{(ax;q)_n (q;q)_n }{(a;q)_n (xq^{-n};q)_{n+1}},\\
\label{simpleNewton}
\sum_{j=0}^{m} (-1)^{m-j} {m\brack j} q^{{j\choose 2}} \frac{(q;q)_{m-j} } {(x;q)_{m-j+1}}
&=\frac{q^{m+1\choose 2}}{1-xq^m}.
\end{align}
\end{lem}
\begin{proof}
For \eqref{simpleLag}, by the partial fraction decomposition we have
$$
\frac{(ax;q)_n (q;q)_n }{(a;q)_n (xq^{-n};q)_{n+1}}=\sum_{k=0}^n\frac{a_k}{1-xq^{-k}}
$$
with
\begin{align*}
a_k&=\lim_{x\to q^k}\frac{(1-xq^{-k})(ax;q)_n (q;q)_n }{(a;q)_n (xq^{-n};q)_{n+1}}=(-1)^{n-k}q^{n-k+1\choose 2}{n\brack k}\frac{(aq^n; q)_k}{(a;q)_k}.
\end{align*}

By the Gauss or $q$-binomial inversion (see, for example, \cite[p.~77, Exercise 2.47]{Aigner}), the identity \eqref{simpleNewton} is equivalent to
\begin{align}
\frac{(q;q)_m}{(x;q)_{m+1}}=\sum_{k=0}^m{m\brack k}_q (-1)^kq^{k+1\choose 2}\frac{1}{1-xq^k},
\end{align}
which is routine by the partial fraction decomposition.
\end{proof}
\begin{lem}Let $n$ be a positive integer. Then
\begin{align}
(x;q)_n+(a/x;q)_n &=(x;q)_n(a/x;q)_n  \nonumber \\
&\quad{}+\sum_{k=0}^{n-1} \frac{(x;q)_k (a/x;q)_k (1-q^n) }{(q;q)_k (1-q^{n-k})}
\sum_{j=0}^k (-1)^j {k\brack j} q^{j\choose 2} (aq^{k+j};q)_{n-k},  \label{eq:lem-important}   \\
(x;q)_n+(a/x;q)_n &=(x;q)_n(a/x;q)_n+(a;q)_n  \nonumber \\
&\hskip -5mm+\sum_{k=1}^{n-1} (x;q)_k (a/x;q)_k (1-q^n)
\sum_{j=1}^{n-k} (-1)^j {n-k-1\brack j-1}{k+j-1\brack j-1}\frac{ q^{{j\choose 2}+kj}a^j}{1-q^j}. \label{eq:lem-important2}
\end{align}
\end{lem}
\pf We first prove \eqref{eq:lem-important}. Taking $x=q^{-m}$  ($0\leqslant m\leqslant n-1$), we have
\begin{align}
&\hskip -2mm\sum_{k=0}^{n-1} \frac{(x;q)_k (a/x;q)_k (1-q^n) }{(q;q)_k (1-q^{n-k})}
\sum_{j=0}^k (-1)^j {k\brack j} q^{j\choose 2} (aq^{k+j};q)_{n-k} \notag\\
& =\sum_{k=0}^{n-1} \frac{(q^{-m};q)_k (aq^m;q)_k (1-q^n) }{(q;q)_k (1-q^{n-k})}
\sum_{j=0}^k (-1)^j {k\brack j} q^{j\choose 2} (aq^{k+j};q)_{n-k}  \notag\\
&=\sum_{k=0}^{m} (-1)^k {m\brack k} q^{{k\choose 2}-mk} \frac{(aq^m;q)_k (1-q^n) }{ (1-q^{n-k})}
\sum_{j=0}^k (-1)^j {k\brack j} q^{j\choose 2} \frac{(aq^j;q)_n}{(aq^j;q)_k}  \notag \\
&=\sum_{j=0}^{m} (-1)^j {m\brack j} q^{{j\choose 2}} (aq^j;q)_n
\sum_{k=j}^{m} (-1)^{k}{m-j\brack k-j} q^{{k\choose 2}-mk} \frac{(aq^m;q)_k (1-q^n)}{(aq^j;q)_k (1-q^{n-k})}. \label{eq:lem-new}
\end{align}
It follows from \eqref{simpleLag} that
\begin{align*}
&\hskip -2mm\sum_{k=j}^{m} (-1)^{k}{m-j\brack k-j} q^{{k\choose 2}-mk} \frac{(aq^m;q)_k (1-q^n)}{(aq^j;q)_k (1-q^{n-k})} \\
&=\frac{(aq^m;q)_j}{(aq^j;q)_j}\sum_{k=j}^{m} (-1)^{k}{m-j\brack k-j} q^{{m-k+1\choose 2}-{m+1\choose 2}}
\frac{(aq^{m+j};q)_{k-j} (1-q^n)}{(aq^{2j};q)_{k-j} (1-q^{n-k})} \\
&=\frac{(-1)^m(aq^m;q)_j (aq^{n+j};q)_{m-j} (q;q)_{m-j}(1-q^n) q^{-{m+1\choose 2}}}{(aq^j;q)_j (aq^{2j};q)_{m-j} (q^{n-m};q)_{m-j+1}} \\
&=\frac{(-1)^m(a;q)_j (aq^{n+j};q)_{m-j} (q;q)_{m-j}(1-q^n)q^{-{m+1\choose 2}} }{(a;q)_m  (q^{n-m};q)_{m-j+1}}.
\end{align*}
Therefore, the right-hand side of \eqref{eq:lem-new} can be simplified as
\begin{align}
\frac{(a;q)_{m+n}(1-q^n)q^{-{m+1\choose 2}} }{(a;q)_m}
\sum_{j=0}^{m} (-1)^{m-j} {m\brack j} q^{{j\choose 2}} \frac{(q;q)_{m-j} } {(q^{n-m};q)_{m-j+1}}=(aq^m;q)_n,\label{eq:use-newton}
\end{align}
where the last equality follows from
\eqref{simpleNewton}. Noticing that $(q^{-m};q)_n=0$ for $0\leqslant m\leqslant n-1$,
we have proved that both sides of \eqref{eq:lem-important} are equal for $x=q^{-m}$ ($0\leqslant m\leqslant n-1$), and by symmetry,
for $x=aq^{m}$ ($0\leqslant m\leqslant n-1$) too. Furthermore, both sides of \eqref{eq:lem-important} are of
the form $x^{-n}P(x)$ with $P(x)$ being a polynomial in $x$ of degree $2n$ with the leading coefficient $(-1)^n q^{n\choose 2}$.
Hence, they must be identical. This proves \eqref{eq:lem-important}.

By the $q$-binomial theorem (see, for example, \cite[Theorem 3.3]{Andrews}), for $k\geqslant 1$, we have
\begin{align}
\sum_{j=0}^k (-1)^j {k\brack j} q^{j\choose 2} (aq^{k+j};q)_{n-k}
&=\sum_{j=0}^k (-1)^j {k\brack j} q^{j\choose 2} \sum_{i=0}^{n-k}(-1)^i{n-k\brack i} q^{{i\choose 2}+(k+j)i}a^i \notag  \\
&=\sum_{i=0}^{n-k}(-1)^i{n-k\brack i} q^{{i\choose 2}+ik}a^i \sum_{j=0}^k (-1)^j {k\brack j} q^{{j\choose 2}+ij} \notag  \\
&=\sum_{i=1}^{n-k}(-1)^i {n-k\brack i} (q^i;q)_k q^{{i\choose 2}+ik}a^i.  \label{eq:kgeq1}
\end{align}
Moreover, for $k=0$, the left-hand side of \eqref{eq:kgeq1} is clearly equal to $(a;q)_n$.
Noticing that
\begin{align*}
{n-k\brack i}\frac{ (q^i;q)_k}{(q;q)_k (1-q^{n-k})}={n-k-1\brack i-1}{k+i-1\brack i-1}\frac{1}{1-q^i},
\end{align*}
we complete the proof of \eqref{eq:lem-important2}.
\qed

\begin{lem}Let $n$ and $h$ be positive integers and let $m$ and $s$ be nonnegative integers with $h\leqslant n-m$ and $s\leqslant m$. Then
\begin{align}
&\sum_{j=s}^{m}\sum_{k=s}^n \frac{(q^{-n};q)_{j}  (q^{-n};q)_{k}(x;q)_{j} (x;q)_{k} (q^{j-m-h+1};q)_{h-1}(q^{k-m-h+1};q)_{h-1} (1-q^{k-j}) q^{2j+k}  }
{(q;q)_{j-s} (q;q)_{j+s}(q;q)_{k-s} (q;q)_{k+s}} \notag\\
&=\frac{(q;q)_n^2 (q;q)_{h-1}(x;q)_s (x;q)_{m+h}(q^{s+1}/x;q)_{n-s-h}x^{n-s-h}q^{\frac{m^2+3m-s^2+s}{2}-mn-mh-h^2+h}}
{(-1)^{m-s-1}(q;q)_{m-s}(q;q)_{m+s}(q;q)_{n-s} (q;q)_{n+s}(q;q)_{n-m-h}}. \label{eq:last-lem}
\end{align}
\end{lem}

\pf  Note that both sides of \eqref{eq:last-lem} are polynomial in $x$ of degree $m+n$ with
the same leading coefficient. Therefore, to prove \eqref{eq:last-lem}, it
suffices to prove that both sides have the same roots as polynomials in $x$.
Denote the left-hand side of \eqref{eq:last-lem} by $L_{m,n}(x)$. We first assert that
\begin{align}
L_{m,n}(x)&=\sum_{j=s}^{m}\sum_{k=m+1}^n \frac{(q^{-n};q)_{j}  (q^{-n};q)_{k}(x;q)_{j} (x;q)_{k}  }
{(q;q)_{j-s} (q;q)_{j+s}(q;q)_{k-s} (q;q)_{k+s}} \notag\\
&\quad\times (q^{j-m-h+1};q)_{h-1}(q^{k-m-h+1};q)_{h-1} (1-q^{k-j}) q^{2j+k}. \label{eq:last-lem-2}
\end{align}
In fact, since $(1-q^{j-k})q^{2k+j}=-(1-q^{k-j})q^{k+2j}$, the double sum
 $\sum_{j=s}^{m}\sum_{k=s}^{m}$
for the same summand in \eqref{eq:last-lem-2} is equal to 0.
We now consider the following three cases.
\begin{itemize}

\item If $s\geqslant 1$, then it is easily seen that $(x;q)_s^2$ divides $L_{m,n}(x)$, which means
that the numbers $q^{-r}$ ($0\leqslant r\leqslant {s-1}$) are roots of $L_{m,n}(x)$ with multiplicity $2$.

\item For $x=q^{-r}$ with $s\leqslant r\leqslant m+h-1$, we have
\begin{align*}
L_{m,n}(q^{-r})&=\sum_{j=s}^{m}\sum_{k=s}^n \frac{(q^{-n};q)_{j}  (q^{-n};q)_{k}(q^{-r};q)_{j} (q^{-r};q)_{k}   }
{(q;q)_{j-s} (q;q)_{j+s}(q;q)_{k-s} (q;q)_{k+s}} \notag\\
&\quad{}\times (q^{j-m-h+1};q)_{h-1}(q^{k-m-h+1};q)_{h-1} (1-q^{k-j}) q^{2j+k}.
\end{align*}
If $r\leqslant m$, then $L_{m,n}(q^{-r})=L_{m,r}(q^{-n})=0$
by the antisymmetry of $j$ and $k$ in $L_{m,r}(q^{-n})$. If $r\geqslant m+1$, then $h\geqslant r-m+1$,
i.e., $r-m-h+1\leqslant 0$, and so $(q^{k-m-h+1};q)_{h-1}=0$ for $m+1\leqslant k\leqslant r$. Hence, by \eqref{eq:last-lem-2},
we again get $L_{m,n}(q^{-r})=L_{m,r}(q^{-n})=0$.

\item For $x=q^{r}$ with $s+1\leqslant r\leqslant n-h$, we shall prove that
\begin{align}
\sum_{k=s}^n \frac{(q^{-n};q)_{k}(q^r;q)_{k} (q^{k-m-h+1};q)_{h-1} (1-q^{k-j}) q^{2j+k}  }
{(q;q)_{k-s} (q;q)_{k+s}} =0.  \label{eq:last-lem-3}
\end{align}
In fact, the left-hand side of \eqref{eq:last-lem-3} can be written as
\begin{align}
&\hskip -2mm (q^{-n};q)_s (q^r;q)_s \sum_{k=s}^n \frac{(q^{-n+s};q)_{k-s}(q^{r+s};q)_{k-s} (q^{k-m-h+1};q)_{h-1} (1-q^{k-j}) q^{2j+k}  }
{(q;q)_{k-s} (q;q)_{k+s}}\notag \\
&=(q^{-n};q)_s (q^r;q)_s \sum_{k=0}^{n-s} (-1)^{k}{n-s\brack
k}q^{-(n-s)k+{k\choose 2}} R_k,\label{eq:last-lem-4}
\end{align}
where
$$
R_k=
\frac{(q^{r+s};q)_{k} (q^{k+s-m-h+1};q)_{h-1} (1-q^{k+s-j}) q^{2j+k+s}} {(q;q)_{k+2s}}.
$$
Since
\begin{align*}
\frac{(q^{r+s};q)_{k}}{(q;q)_{k+2s}}=
\frac{(q^{k+2s+1};q)_{r-s-1}} {(q;q)_{r+s-1}},
\end{align*}
we see that $R_k$ is a polynomial in $q^k$ of degree $r-s-1+h-1+2\leqslant n-s$ with constant term $0$.
By the finite $q$-binomial theorem, see \cite[Theorem 3.3]{Andrews},
\begin{align}\label{q-binomial}
\sum_{k=0}^{n}(-1)^{k}{n\brack k}_{q}q^{{k+1\choose 2}} x^{k}
=(xq;q)_n,
\end{align}
we have
\begin{align}\label{eq:q-diff}
\sum_{k=0}^n(-1)^{k}{n\brack k}_{q} q^{{k+1\choose 2}-ik}
=\begin{cases}
0,&\text{for $1\leqslant i\leqslant n$}, \\
(q;q)_n, &\text{for $i=0$}.
\end{cases}
\end{align}

It follows that the right-hand side of  \eqref{eq:last-lem-4} is equal to $0$. Namely, the identity \eqref{eq:last-lem-3} holds.
\end{itemize}
Thus, we have found out all the $m+n$ roots of $L_{m,n}(x)$, which are clearly the same as those of the right-hand side of \eqref{eq:last-lem}.  This completes the proof. \qed

\medskip
\noindent{\it Remark.} The authors \cite{GZ05} utilized the identity \eqref{eq:q-diff} to give a short proof of
Jackson's terminating $q$-analogue of Dixon's identity~\cite{Bailey41,Jackson41}:
\begin{equation*}
\sum_{k=-a}^a(-1)^k q^{\frac{3k^2+k}{2}}
{a+b\brack a+k}{b+c\brack b+k}{c+a\brack c+k}={a+b+c\brack a+b}{a+b\brack a}.
\end{equation*}

\begin{lem}
Let $n$ and $h$ be positive integers and let $m$ and $s$ be nonnegative integers with $h\leqslant n-m$ and $s\leqslant m$. Then
\begin{align}
&\hskip -2mm\sum_{j=s}^{m}\sum_{k=m+h}^n \frac{(q^{-2n};q^2)_j (q^{-2n};q^2)_k (1-q^{k-j}) q^{j+k+jh}  }
{(q;q)_{j-s} (q;q)_{j+s}(q;q)_{k-s} (q;q)_{k+s}}
{k-m-1\brack h-1}{m+h-j-1\brack h-1} \notag\\
&=\frac{(-1)^{n-m-h}(q^2;q^2)_n^2 (-q;q)_{2n-h} q^{m^2-n^2-2mn+mh}}
{(q;q)_{m-s}(q;q)_{m+s}(q^2;q^2)_{n-s} (q^2;q^2)_{n+s}(q;q)_{h-1}(q^2,q^2)_{n-m-h}}. \label{eq:am-2}
\end{align}
\end{lem}
\pf By the definition of $q$-binomial coefficients, there holds ${k-m-1\brack h-1}=0$ for $m+1\leqslant k<m+h$.
Hence, the left-hand side of \eqref{eq:am-2} remains unchanged if we replace $\sum_{k=m+h}^n$ by $\sum_{k=m+1}^n$.
Furthermore, we have
\begin{align*}
{k-m-1\brack h-1}{m+h-j-1\brack h-1}
=\frac{(q^{j-m-h+1};q)_{h-1}(q^{k-m-h+1};q)_{h-1}q^{(m-j)(h-1)-{h\choose 2}} }{(-1)^{h-1}(q;q)_{h-1}^2}
\end{align*}
The proof then follows from the identity \eqref{eq:last-lem} with $x=-q^{-n}$.
\qed

\noindent{\it Proof of Theorem~\ref{thm:important-thm}.}
 The left-hand side of \eqref{eq:important-thm} may be expanded as
\begin{align}
&\sum_{k=s}^{n} \frac{(q^{-2n};q^2)_k^2 (x;q)_k^2 q^{2k} }{(q;q)_{k-s}^2 (q;q)_{k+s}^2}(x;q)_k (q/x;q)_k \notag\\
&\quad{}+\sum_{s\leqslant j<k\leqslant n}
\frac{(q^{-2n};q^2)_j (q^{-2n};q^2)_k  q^{j+k} \big((x;q)_j (q/x;q)_k+(x;q)_k (q/x;q)_j \big)}
{(q;q)_{j-s} (q;q)_{j+s}(q;q)_{k-s} (q;q)_{k+s} }.  \label{eq:thm-pf-1}
\end{align}
For $0\leqslant j<k$, from \eqref{eq:lem-important2} we deduce that
\begin{align}
&\hskip -2mm (x;q)_j (q/x;q)_k+(x;q)_k (q/x;q)_j \notag\\
&=(x;q)_j(q/x;q)_j\big((xq^j;q)_{k-j}+(q^{j+1}/x;q)_{k-j}\big) \notag \\
&=(x;q)_k(q/x;q)_k+(x;q)_j(q/x;q)_j(q^{2j+1};q)_{k-j}  \notag\\
&\quad{}+\sum_{i=1}^{k-j-1} (x;q)_{j+i} (q/x;q)_{j+i} (1-q^{k-j})
\sum_{h=1}^{k-j-i} (-1)^h {k-j-i-1\brack h-1}{i+h-1\brack h-1} \frac{ q^{{h+1\choose 2}+(i+2j)h} }{1-q^h}, \notag\\
&=(x;q)_k(q/x;q)_k+(x;q)_j(q/x;q)_j \notag\\
&\quad{}+\sum_{i=0}^{k-j-1} (x;q)_{j+i} (q/x;q)_{j+i} (1-q^{k-j})
\sum_{h=1}^{k-j-i} (-1)^h {k-j-i-1\brack h-1}{i+h-1\brack h-1} \frac{ q^{{h+1\choose 2}+(i+2j)h} }{1-q^h}, \label{eq:lem-tothm}
\end{align}
where in the last step we have used the $q$-binomial theorem:
$$
(q^{2j+1};q)_{k-j}=1+\sum_{h=1}^{k-j}(-1)^h{k-j\brack h}q^{{h+1\choose 2}+2jh}.
$$
By \eqref{eq:lem-tothm}, we may write \eqref{eq:thm-pf-1} as
$
\sum_{m=s}^n a_m (x;q)_m (q/x;q)_m,
$
where
\begin{align}
a_m&=\sum_{j=s}^n\frac{(q^{-2n};q^2)_j (q^{-2n};q^2)_m  q^{j+m}}
{(q;q)_{j-s} (q;q)_{j+s}(q;q)_{m-s} (q;q)_{m+s} }  \notag\\
&\quad{}+\sum_{j=s}^{m}\sum_{k=m+1}^n \frac{(q^{-2n};q^2)_j (q^{-2n};q^2)_k (1-q^{k-j}) q^{j+k}  }
{(q;q)_{j-s} (q;q)_{j+s}(q;q)_{k-s} (q;q)_{k+s}} \notag \\
&\qquad\qquad{}\times \sum_{h=1}^{k-m} (-1)^h {k-m-1\brack h-1}{m+h-j-1\brack h-1} \frac{ q^{{h+1\choose 2}+(m+j)h} }{1-q^h}. \label{eq:am-0}
\end{align}

It is easy to see that
\begin{align*}
\sum_{j=s}^n\frac{(q^{-2n};q^2)_j q^{j}}
{(q;q)_{j-s} (q;q)_{j+s}}
&=\frac{(q^{-2n};q^2)_s q^{s}}{(q;q)_{2s}}\sum_{j=s}^n\frac{(q^{-2n+2s};q^2)_j q^{j-s}}
{(q;q)_{j-s} (q^{2s+1};q)_{j-s}}  \\
&=\frac{(q^{-2n};q^2)_s q^{s}}{(q;q)_{2s} } {}_{2}\phi_{1}\left[\!\!\begin{array}{c}
q^{-n+s},\, -q^{-n+s},\\
q^{2s+1}\end{array}\!\!;q,q\right] \\
&=(-1)^{n-s}\frac{(q^{-2n};q^2)_s (-q^{n+s+1};q)_{n-s} q^{s-(n-s)^2}}{(q;q)_{2s} (q^{2s+1};q)_{n-s} }
\end{align*}
by the $q$-Chu-Vandermonde summation formula \cite[Appendix (II.6)]{GR}. Hence,
\begin{align}
&\hskip -2mm \sum_{j=s}^n\frac{(q^{-2n};q^2)_j (q^{-2n};q^2)_m  q^{j+m}}
{(q;q)_{j-s} (q;q)_{j+s}(q;q)_{m-s} (q;q)_{m+s} } \nonumber \\
&=(-1)^{n-s} \frac{(q^{-2n};q^2)_s (q^{-2n};q^2)_m  (-q^{n+s+1};q)_{n-s} q^{m+s-(n-s)^2}}{(q;q)_{m-s}(q;q)_{m+s} (q;q)_{n+s} } \nonumber \\
&=\frac{(-1)^{n-m}(q^2;q^2)_n^2 (-q;q)_{2n} q^{m^2-n^2-2mn}}
{(q;q)_{m-s}(q;q)_{m+s}(q^2;q^2)_{n-s} (q^2;q^2)_{n+s}(q^2,q^2)_{n-m}}. \label{eq:am-1}
\end{align}
Substituting \eqref{eq:am-1} and \eqref{eq:am-2} into \eqref{eq:am-0}, we obtain
\begin{align}
a_m=\frac{(-1)^{n-m}(q^2;q^2)_n^2 q^{m^2-n^2-2mn}}{(q;q)_{m-s}(q;q)_{m+s}(q^2;q^2)_{n-s} (q^2;q^2)_{n+s}}
\sum_{h=0}^{n-m}\frac{(-1)^h(-q;q)_{2n-h} q^{{h+1\choose 2}+2mh}}
{(q;q)_{h}(q^2,q^2)_{n-m-h}}. \label{eq:fin-1}
\end{align}
Replacing $h$ by $n-m-h$, we have
\begin{align}
&\sum_{h=0}^{n-m}\frac{(-1)^h(-q;q)_{2n-h} q^{{h+1\choose 2}+2mh}}
{(q;q)_{h}(q^2,q^2)_{n-m-h}} \notag\\
&=
\frac{(-q;q)_{n+m}}{(q;q)_{n-m}}(-1)^{n-m} q^{{n-m+1\choose 2}+2m(n-m)}\sum_{h=0}^{n-m}
\frac{(q^{m-n};q)_h(-q^{n+m+1}; q)_h}{(-q;q)_h(q;q)_h}q^{-2hm} \notag\\
&=\frac{(-q;q)_{n+m}}{(q;q)_{n-m}}(-1)^{n-m} q^{{n-m+1\choose 2}+2m(n-m)}
{}_{2}\phi_{1}\left[\!\!\begin{array}{c}
q^{-(n-m)},\, -q^{m+n+1}\\
-q\end{array}\!\!;q,q^{-2m}\right] \notag \\
&=\frac{(-q;q)_{n+m} (q^{-n-m};q)_{n-m} }{(q;q)_{n-m} (-q;q)_{n-m}}(-1)^{n-m}
q^{{n-m+1\choose 2}+2m(n-m)} \notag\\
&=\frac{(q^2;q^2)_{m+n}}{(q^2;q^2)_{n-m}(q;q)_{2m}},  \label{eq:fin-2}
\end{align}
where we have used the $q$-Chu-Vandermonde summation formula \cite[Appendix (II.7)]{GR}.
It follows from \eqref{eq:fin-1} and \eqref{eq:fin-2} that $a_m$ is just the coefficient of $(x;q)_m (q/x;q)_m$ in the right-hand side of
\eqref{eq:important-thm}. This completes the proof.
\qed

\section{Proof of Theorem \ref{thm:2k2k2k-2} }
We first give a congruence modulo $[p]$.
\begin{lem}\label{thm:factor-1}
Let $p$ be an odd prime and $m$, $r$ two positive integers with $p\nmid m$.
Let $s\leqslant\min\{\langle -\frac{r}{m}\rangle_p,\langle -\frac{m-r}{m}\rangle_p\}$ be
a nonnegative integer. Then the following congruence holds modulo $[p]${\rm:}
\begin{align}
&\hskip -3mm \sum_{k=s}^{\frac{p-1}{2}}\frac{(q^{m};q^{2m})_{k}(q^r;q^m)_k q^{mk} }{ (q^m;q^m)_{k-s}(q^m;q^m)_{k+s} } \nonumber \\
&\equiv
\begin{cases}
\displaystyle\frac{q^{\frac{(\langle -\frac{r}{m}\rangle_p+s)m}{2}}(q^m; q^{2m})_{\frac{\langle -\frac{r}{m}\rangle_p-s}{2}}
(q^{-m\langle -\frac{r}{m}\rangle_p};q^m)_s}
{(q^{2m}; q^{2m})_{\frac{\langle -\frac{r}{m}\rangle_p+s}{2}}},
&\text{if $\langle -\frac{r}{m}\rangle_p\equiv s\pmod 2$,} \\[5pt]
0 &\text{if $\langle -\frac{r}{m}\rangle_p\equiv s+1\pmod 2$.}
\end{cases} \label{eq:factor-001}
\end{align}
\end{lem}

\pf It is easy to see that $\langle -\frac{r}{m}\rangle_p+\langle -\frac{m-r}{m}\rangle_p=p-1$, and so $s\leqslant\frac{p-1}{2}$.
Since $p$ is an odd prime, we see that $(q^m;q^{2m})_k\equiv 0\pmod{[p]}$ for $\frac{p+1}{2}\leqslant k\leqslant p-s-1$,
which means that
\begin{align*}
\sum_{k=s}^{\frac{p-1}{2}}\frac{(q^{m};q^{2m})_{k}(q^r;q^m)_k q^{mk} }{ (q^m;q^m)_{k-s}(q^m;q^m)_{k+s}  }
\equiv \sum_{k=s}^{p-s-1}\frac{(q^{m};q^{2m})_{k}(q^r;q^m)_k q^{mk} }{ (q^m;q^m)_{k-s}(q^m;q^m)_{k+s}  } \pmod{[p]}.
\end{align*}
Let
$a=\frac{m\langle -\frac{r}{m}\rangle_p+r}{p}$. Then $m\!\mid\! r-ap$ and $r-ap=-m\langle -\frac{r}{m}\rangle_p\leqslant 0$. It is clear that $(q^r;q^m)_k\equiv (q^{r-ap};q^m)_k\pmod{[p]}$
and $(q^{r-ap};q^m)_k=0$ for $k>\langle -\frac{r}{m}\rangle_p$. Moreover,
we have $p-s-1\geqslant p-\langle -\frac{m-r}{m}\rangle_p-1=\langle -\frac{r}{m}\rangle_p\geqslant s$,
and therefore,
\begin{align*}
\sum_{k=s}^{p-s-1}\frac{(q^{m};q^{2m})_{k}(q^r;q^m)_k q^{mk} }{ (q^m;q^{m})_{k-s}(q;q)_{k+s} }
&\equiv \sum_{k=s}^{p-s-1}\frac{(q^{m};q^{2m})_{k}(q^{r-ap};q^m)_k q^{mk} }{ (q^m;q^{m})_{k-s}(q;q)_{k+s} } \\
&=\sum_{k=s}^{\langle -\frac{r}{m}\rangle_p}\frac{(q^{m};q^{2m})_{k}(q^{r-ap};q^m)_k q^{mk} }{ (q^m;q^m)_{k-s}(q^m;q^m)_{k+s} }  \\
&=\frac{(q^{m};q^{2m})_{s}(q^{-m\langle -\frac{r}{m}\rangle_p};q^m)_s q^{ms} }{(q^m;q^m)_{2s}  }  \\
&\quad{}\times {}_{3}\phi_{2}\left[\!\!\!\begin{array}{c}
q^{-m(\langle -\frac{r}{m}\rangle_p-s)}, q^{(s+\frac{1}{2})m}, -q^{(s+\frac{1}{2})m}\\
0, q^{(2s+1)m} \end{array}\!\!\!;q^m,q^m\right]
\pmod{[p]}.
\end{align*}
The proof then follows from Andrew's identity \eqref{eq:andrews}.
\qed

\medskip
\noindent{\it Proof of Theorems {\rm\ref{thm:2k2k2k-2}.}}
By Lemma \ref{lem:one}, for $0\leqslant k\leqslant \frac{p-1}{2}$, we have
\begin{align*}
\frac{(q^m;q^m)_{2k}}{(q^{2m};q^{2m})_{k}^2}&={2k\brack k}_{q^{2m}}\frac{1}{(-q^m;q^m)_{2k}}  \\
&\equiv (-1)^k q^{mk^2-mkp}{\frac{p-1}{2}+k\brack 2k}_{q^{2m}}(-q^m;q^m)_{2k}   \\
&=\frac{(-1)^k q^{mk^2-mkp}(q^{2m};q^{2m})_{\frac{p-1}{2}+k}}{(q^{2m};q^{2m})_{\frac{p-1}{2}-k}(q^m;q^m)_{2k}}  \pmod{[p]^2},
\end{align*}
and so
\begin{align}
&\hskip -2mm\sum_{k=s}^{\frac{p-1}{2}}\frac{(q^m;q^m)_{2k}(q^r;q^m)_k (q^{m-r};q^m)_{k} q^{mk} }{ (q^m;q^m)_{k-s}(q^m;q^m)_{k+s} (q^{2m};q^{2m})_k^2 } \notag\\
&\equiv \sum_{k=s}^{\frac{p-1}{2}}\frac{(-1)^k(q^m;q^m)_{\frac{p-1}{2}+k}(q^r;q^m)_k (q^{m-r};q^m)_{k} q^{mk^2-mk(p-1)} }
    {(q^{2m};q^{2m})_{\frac{p-1}{2}-k}(q^m;q^m)_{k-s}(q^m;q^m)_{k+s} (q^{m};q^{m})_{2k} } \pmod{[p]^2}. \label{eq:pf-thm-1}
\end{align}

Letting $q\to q^m$, $x=q^{r}$ and $n=\frac{p-1}{2}$ in Theorem \ref{thm:important-thm},
we see that the right-hand side of \eqref{eq:pf-thm-1} can be written as
\begin{align}
&\frac{(-1)^{\frac{p-1}{2}}(q^{2m};q^{2m})_{\frac{p-1}{2}-s}(q^{2m};q^{2m})_{\frac{p-1}{2}+s} q^{\frac{(p-1)^2}{4}}}
{(q^{2m};q^{2m})_{\frac{p-1}{2}}^2}
\left(\sum_{k=s}^{\frac{p-1}{2}} \frac{(q^{m(1-p)};q^{2m})_k (q^r;q^m)_k q^{mk} }{(q^m;q^m)_{k-s} (q^m;q^m)_{k+s}}\right) \notag \\
&\quad\times\left(\sum_{k=s}^{\frac{p-1}{2}} \frac{(q^{m(1-p)};q^{2m})_k (q^{m-r};q^m)_k q^{mk} }{(q^m;q^m)_{k-s} (q^m;q^m)_{k+s}}\right).
\label{eq:factor-main}
\end{align}
If $\langle -\frac{r}{m}\rangle_p\equiv s+1\pmod 2$, then by the congruence \eqref{eq:factor-001}, we have
\begin{align*}
\sum_{k=s}^{\frac{p-1}{2}} \frac{(q^{m(1-p)};q^{2m})_k (q^r;q^m)_k q^{mk} }{(q^m;q^m)_{k-s} (q^m;q^m)_{k+s}}
\equiv \sum_{k=s}^{\frac{p-1}{2}} \frac{(q^{m};q^{2m})_k (q^r;q^m)_k q^{mk} }{(q^m;q^m)_{k-s} (q^m;q^m)_{k+s}}
\equiv 0 \pmod{[p]},
\end{align*}
and also $\langle -\frac{m-r}{m}\rangle_p\equiv s+1\pmod 2$ which means that
\begin{align*}
\sum_{k=s}^{\frac{p-1}{2}} \frac{(q^{m(1-p)};q^{2m})_k (q^{m-r};q^m)_k q^{mk} }{(q^m;q^m)_{k-s} (q^m;q^m)_{k+s}}
\equiv 0 \pmod{[p]}.
\end{align*}
Noticing the fact $(q^{2m};q^{2m})_{\frac{p-1}{2}}\not\equiv 0\pmod{[p]}$, we conclude that
the right-hand side of \eqref{eq:pf-thm-1} is congruent to $0$ modulo $[p]^2$. This proves \eqref{eq:new-2mpp-002}.

To prove \eqref{eq:new-2mpp}, just observe that (see the proof of Lemma \ref{thm:factor-1})
\begin{align*}
(q^r;q^m)_k \equiv (q^{m-r};q^m)_{k} &\equiv 0\pmod{[p]}
\end{align*}
for $\max\left\{\langle -\frac{r}{m}\rangle_p,\langle -\frac{m-r}{m}\rangle_p\right\}
<k\leqslant p-1$, and
\begin{align*}
(q^r;q^m)_k (q^{m-r};q^m)_{k} &\equiv \frac{(q^m;q^m)_{2k}}{(q^m;q^m)_{k-s}(q^m;q^m)_{k+s}}\equiv 0 \pmod{[p]}.
\end{align*}
for $\frac{p-1}{2}
< k \leqslant\max\left\{\langle -\frac{r}{m}\rangle_p,\langle -\frac{m-r}{m}\rangle_p\right\}.$

Finally, the proof of \eqref{eq:new-2mpp-003} follows from factorizing \eqref{eq:pf-thm-1} into \eqref{eq:factor-main},
applying the first case of the congruence \eqref{eq:factor-001}, and then using the aforementioned relation
$\langle -\frac{r}{m}\rangle_p+\langle -\frac{m-r}{m}\rangle_p=p-1$.
\qed

\section{Proof of Theorems \ref{thm:2k2k-r} and \ref{thm:2mp2}}
The following lemma can be derived from
the  $q$-Chu-Vandermonde formula if the sums are written in
standard basic hypergeometric series. Here we give a different proof.
\begin{lem}\label{lem:two} Let $n$ and $s$ be nonnegative integers with $s\leqslant n$. Then
\begin{align}
\sum_{k=0}^n (-1)^k {n\brack k} {m+k\brack n} q^{{k\choose 2}-nk}
&=(-1)^n q^{-{n+1\choose 2}},  \label{eq:lem-qbino}\\
\sum_{k=0}^n (-1)^k {n+k\brack 2k}_{q^2} {2k+2s\brack k+s}_{q^2}\frac{q^{k^2-k-2nk}}{(-q^{2k+1};q)_{2s} }
&=(-1)^n q^{-n(n+1)}. \label{eq:lemtwo}
\end{align}
\end{lem}
\pf
It is not difficult to see that  the two identities \eqref{eq:lem-qbino} and \eqref{eq:lemtwo}
are equivalent, respectively,  to
\begin{align}
\sum_{k=0}^n (-1)^{k} {n\brack k} {m+n-k\brack n} q^{{k+1\choose 2}}&=1,\label{eq:equiv1}\\
\sum_{k=0}^n (-1)^{k} {n\brack k}_{q^2}
{2n-k\brack n}_{q^2}
\frac{(q^{2n-2k+1};q^2)_{s}}{(q^{2n-2k+2};q^2)_s }q^{2{k+1\choose 2}}&=1.\label{eq:equiv2}
\end{align}
Since  ${m+n-k\brack n}$ can be written as a polynomial in $q^{-k}$
of degree $n$ with constant term $1/(q;q)_n$. Identity~\eqref{eq:equiv1} then follows from \eqref{eq:q-diff}.
On the other hand,
since $0\leqslant s\leqslant n$, we see that
\begin{align*}
{2n-k\brack n}_{q^2}
\frac{(q^{2n-2k+1};q^2)_{s}}{(q^{2n-2k+2};q^2)_s }
=\frac{(q^{2n-2k+2s+2};q^2)_{n-s} (q^{2n-2k+1};q^2)_{s}}{(q^2;q^2)_n}
\end{align*}
is a polynomial in $q^{-2k}$ of degree $n$ with constant term $\frac{1}{(q^2;q^2)_n}$.
Therefore,  identity~\eqref{eq:equiv2} follows from \eqref{eq:q-diff} with
$q\rightarrow q^2$. This completes the proof.  \qed

\medskip
\noindent{\it Proof of Theorem {\rm\ref{thm:2k2k-r}.}} It is easy to see that
\begin{align*}
\frac{(q;q^2)_k }{(q^2;q^2)_k }={2k\brack k}_{q^2}\frac{1}{(-q;q)_{2k}}.
\end{align*}
Hence, by Lemmas \ref{lem:one} and \ref{lem:two}, we have
\begin{align*}
\sum_{k=0}^{\frac{p-1}{2}}\frac{(q;q^2)_k (q;q^2)_{k+s}}{(q^2;q^2)_k (q^2;q^2)_{k+s} }
&=\sum_{k=0}^{\frac{p-1}{2}}{2k\brack k}_{q^2} {2k+2s\brack k+s}_{q^2}\frac{1}{(-q;q)_{2k}(-q;q)_{2k+2s}} \\
&\equiv \sum_{k=0}^{\frac{p-1}{2}} (-1)^k {\frac{p-1}{2}+k\brack 2k}_{q^2} {2k+2s\brack k+s}_{q^2}\frac{(-q;q)_{2k} q^{k^2-kp}}{(-q;q)_{2k+2s}}\\
&=(-1)^{\frac{p-1}{2}} q^{\frac{1-p^2}{4}} \pmod{[p]^2},
\end{align*}
as desired.  \qed



\noindent{\it Proof of Theorem {\rm\ref{thm:2mp2}.}}
Again, let $a=\frac{m\langle -\frac{r}{m}\rangle_p+r}{p}$.
Then $a$ is a positive integer, $m\!\mid\! ps-r$, and so
\begin{align}
\frac{(q^r;q^m)_k }{(q^m;q^m)_k}
&=\prod_{j=1}^{k+s} \frac{1-q^{mj+r-m}}{1-q^{mj}} \notag\\
&\equiv (-1)^k\prod_{j=1}^k \frac{(1-q^{ap-mj-r+m})q^{mj+r-m}}{1-q^{mj}} \notag  \\
&= (-1)^k {\frac{ap-r}{m}\brack k}_{q^m} q^{\frac{mk(k-1)}{2}+kr} \notag\\
&\equiv (-1)^{k}{\langle -\frac{r}{m}\rangle_p\brack k}_{q^m} q^{\frac{mk(k-1)}{2}-mk\langle -\frac{r}{m}\rangle_p}
\pmod{[p]},  \label{eq:frac-bino-1}  \\[5pt]
\frac{(q^{m-r};q^m)_{k+s} }{(q^m;q^m)_{k+s}}
&\equiv \prod_{j=1}^{k+s} \frac{1-q^{ap+mj-r}}{1-q^{mj}}
= {\frac{ap-r}{m}+k+s\brack k+s}_{q^m}
= {\langle -\frac{r}{m}\rangle_p+k+s\brack k+s}_{q^m}
\pmod{[p]}.  \label{eq:frac-bino-2}
\end{align}
By the congruences \eqref{eq:frac-bino-1} and \eqref{eq:frac-bino-2}, we have
\begin{align*}
&\hskip -2mm\sum_{k=0}^{p-s-1}\frac{(q^r;q^m)_k (q^{m-r};q^m)_{k+s} }{(q^m;q^m)_k (q^m;q^m)_{k+s} }  \\
&\equiv \sum_{k=0}^{p-s-1} (-1)^{k}{\langle -\frac{r}{m}\rangle_p\brack k}_{q^m}
{\langle -\frac{r}{m}\rangle_p+k+s\brack k+s}_{q^m}q^{\frac{mk(k-1)}{2}-mk\langle -\frac{r}{m}\rangle_p} \\
&=(-1)^{\langle -\frac{r}{m}\rangle_p}
q^{\frac{-m\langle -\frac{r}{m}\rangle_p \left(\langle -\frac{r}{m}\rangle_p+1 \right)}{2}} \pmod{[p]},
\end{align*}
where in the last step we have used $p-s-1\geqslant p-\langle -\frac{m-r}{m}\rangle_p-1=\langle -\frac{r}{m}\rangle_p$
and the identity \eqref{eq:lem-qbino}. This proves \eqref{eq:new-2mp}.

To prove \eqref{eq:new-2mp-2}, just notice that if $p\equiv \pm 1\pmod m$, then
$\frac{r(m-r)(1-p^2)}{2m}$ is an integer and
\begin{align}
\frac{-m\langle -\frac{r}{m}\rangle_p(\langle -\frac{r}{m}\rangle_p+1)}{2}
\equiv \frac{r(m-r)(1-p^2)}{2m}\pmod{p}.  \tag*{$\square$}
\end{align}

\section{Concluding remarks and open problems}
It is natural to ask the following problem:
\begin{prob}Are there any $q$-analogues of Beukers' supercongruence \eqref{eq:beukers-0}?
\end{prob}

Numerical experiments suggest the following companion of Theorem \ref{thm:important-thm}.
\begin{conj}Let $n$ and $r$ be nonnegative integers with $r\leqslant n$. Then
\begin{align*}
&\hskip -3mm\left( \sum_{k=r}^{n} \frac{(q^{-n};q)_k (x;q^2)_k q^k }{(q;q)_{k-r} (q;q)_{k+r}} \right)
  \left( \sum_{k=r}^{n} \frac{(q^{-n};q)_k (x;q^2)_k q^{(n+1)k-{k\choose 2}}} {(q;q)_{k-r} (q;q)_{k+r} x^k} \right) \nonumber \\
&=\frac{(-1)^r(q;q)_n^2 (x;q^2)_r q^r}{(q;q)_{n-r}(q;q)_{n+r}(q^2/x;q^2)_r x^r}
\sum_{k=r}^{n} \frac{(q^{-n};q)_k (q^{n+1};q)_k (x;q^2)_k (q^2/x;q^2)_k  q^{k}}{(q;q)_{k-r}(q;q)_{k+r} (q;q)_{2k}}.
\end{align*}
\end{conj}

It seems that the congruence \eqref{eq:new-2mpp} can be further generalized as follows.
\begin{conj}\label{conj:2k2k2k-1}
Let $p$ be an odd prime and $m$, $r$ two positive integers with $p\nmid m$.
Let $s\leqslant p-1$ be a nonnegative integer. If $\langle -\frac{r}{m}\rangle_p\equiv s+1\pmod 2$, then
\begin{align}
\sum_{k=s}^{p-1}\frac{(q^m;q^m)_{2k}(q^r;q^m)_k (q^{m-r};q^m)_{k} q^{mk} }{ (q^m;q^m)_{k-s}(q^m;q^m)_{k+s} (q^{2m};q^{2m})_k^2 }
\equiv 0 \pmod{[p]^2}. \label{eq:new-2mp-333}
\end{align}
\end{conj}

Note that, if $s>\min\{\langle -\frac{r}{m}\rangle_p, \langle -\frac{m-r}{m}\rangle_p\}$, then the congruence \eqref{eq:new-2mp-333}
is obviously true, since in this case each summand in the left-hand side is congruent to $0$ modulo $[p]^2$.

Taking $r=1$ and $m=3,4,6$ in \eqref{eq:new-2mp-333}, we get
\begin{conj}\label{conj:sunzw}
Let $p\geqslant 5$ be a prime and let $s\leqslant p-1$ be a nonnegative integer. Then
\begin{align*}
\sum_{k=s}^{p-1}{2k\brack k+s}_{q^3}\frac{(q;q^3)_k (q^{2};q^3)_{k} q^{3k} }{ (q^{6};q^{6})_k^2 }
&\equiv 0 \pmod{[p]^2},\quad\text{if }s\equiv \frac{1+(\frac{-3}{p})}{2}\pmod{2},\\
\sum_{k=s}^{p-1}{2k\brack k+s}_{q^4}\frac{(q;q^4)_k (q^{3};q^4)_{k} q^{4k} }{ (q^{8};q^{8})_k^2 }
&\equiv 0 \pmod{[p]^2},\quad\text{if }s\equiv \frac{1+(\frac{-2}{p})}{2}\pmod{2},\\
\sum_{k=s}^{p-1}{2k\brack k+s}_{q^6}\frac{(q;q^6)_k (q^{5};q^6)_{k} q^{6k} }{ (q^{12};q^{12})_k^2 }
&\equiv 0 \pmod{[p]^2},\quad\text{if }s\equiv \frac{1+(\frac{-1}{p})}{2}\pmod{2}.
\end{align*}
\end{conj}

It is clear that when $q\to 1$, Conjecture \ref{conj:sunzw} reduces to Z.-W. Sun's generalization of
\eqref{eq:RV-B1}--\eqref{eq:RV-B3} (see \cite[Theorem 1.3(i)]{Sun0}).

We conjecture  that Theorem \ref{thm:2mp2} and Corollary \ref{cor:two} can be further strengthened.
\begin{conj}
The congruences \eqref{eq:new-RV2}--\eqref{eq:new-RV4} also hold modulo $[p]^2$.
\end{conj}

\begin{conj}\label{conj:fifi}
Let $p$ be an odd prime and let $m$, $r$ be positive integers with $p\equiv\pm 1\pmod{m}$ and $r<m$.
Then for any integer $s$ with $0\leqslant s\leqslant \langle -\frac{m-r}{m}\rangle_p$,
there holds
\begin{align*}
\sum_{k=0}^{p-s-1}\frac{(q^r;q^m)_k (q^{m-r};q^m)_{k+s} }{(q^m;q^m)_k (q^m;q^m)_{k+s} }
\equiv (-1)^{\langle -\frac{r}{m}\rangle_p}q^{\frac{r(m-r)(1-p^2)}{2m}} \pmod{[p]^2}. 
\end{align*}
\end{conj}

Like \cite[Conjecture 7.1]{GZ}, Conjecture \ref{conj:fifi} seems have a further generalization as follows:
\begin{conj}
Let $p$ be an odd prime and let $m$, $|r|$ be positive integers with $p\nmid m$ and $m\nmid r$.
Then there exists a unique integer $f_{p,m,r}$ such that, for any $0\leqslant s\leqslant \langle -\frac{m-r}{m}\rangle_p$, there holds
\begin{align*}
\sum_{k=0}^{p-s-1}\frac{(q^r;q^m)_k (q^{m-r};q^m)_{k+s} }{(q^m;q^m)_k (q^m;q^m)_{k+s} }
\equiv (-1)^{\langle -\frac{r}{m}\rangle_p}q^{f_{p,m,r}} \pmod{[p]^2}.
\end{align*}
Furthermore, the numbers $f_{p,m,r}$ satisfy the symmetry $f_{p,m,r}=f_{p,m,m-r}$ and the recurrence relation{\rm:}
\begin{align*}
f_{p,m,m+r}
=\begin{cases}
-f_{p,m,r},&\text{if $r\equiv 0\pmod{p},$}\\[5pt]
f_{p,m,r}-r, &\text{otherwise.}
\end{cases}
\end{align*}
\end{conj}

Here we give some values of $f_{p,m,r}$:
\begin{align*}
&f_{3,2,1}=-2,\ f_{3,2,3}=-3,\ f_{3,2,5}=3,\  f_{3,2,7}=-2,\
f_{3,2,9}=-9,\ f_{3,2,11}=9,\ f_{3,2,13}=-2,\   \\
&f_{5,3,1}=-8,\ f_{5,3,2}=-8,\ f_{5,3,4}=-9,\  f_{5,3,5}=-10,\
f_{5,3,7}=-13,\ f_{5,3,8}=10,\ f_{5,3,10}=-20,\   \\
&f_{5,3,11}=2,\ f_{5,3,13}=20,\  f_{5,3,14}=-9,\  f_{5,3,16}=7,\
f_{5,3,17}=-23,\ f_{5,3,19}=-9,\  f_{5,8,1}=-23,  \\
&f_{7,9,1}=-54,\ f_{7,9,2}=-21,\ f_{7,9,4}=-37,\  f_{7,9,5}=-37,\
f_{7,9,7}=-21,\ f_{7,9,8}=-54,  \\
&f_{7,9,10}=-55,\ f_{7,9,11}=-23\ f_{7,9,13}=-41,\  f_{7,9,14}=-42,\
f_{7,9,16}=-22,\ f_{7,9,17}=-33.
\end{align*}

\vskip 3mm
\noindent{\bf Acknowledgments.} The first author was partially
supported by the Fundamental Research Funds for the Central Universities and
the National Natural Science Foundation of China (grant 11371144).

\end{document}